\documentclass[a4paper,12pt]{article}
\usepackage[utf8]{inputenc}
\usepackage{graphicx, amsmath, amssymb, amsthm,anysize,float,framed,paralist}
\usepackage{todonotes}
\usepackage{algorithm,algorithmic}
\usepackage{hyperref}
\hypersetup{citecolor=red, linkcolor=blue, colorlinks=true}

\marginsize{2cm}{2cm}{1cm}{3cm}

\newtheorem{theorem}{Theorem}
\newtheorem{lemma}[theorem]{Lemma}

\newtheorem{corollary}[theorem]{Corollary}

\begin{document}

\title{Cospectral mates for the union of some classes in the Johnson association scheme}
\author{Sebastian M.Cioab\u{a}\footnote{Department of Mathematical Sciences, University of Delaware, Newark, DE 19716-2553, USA. Research supported by NSF DMS-1600768 grant. {\tt cioaba@udel.edu}}\,, Willem H. Haemers\footnote{Tilburg University, Dept. of Econometrics and OR, Tilburg, The Netherlands. {\tt haemers@uvt.nl}},\\ Travis Johnston\footnote{Oak Ridge National Laboratory, Oak Ridge, TN 37831. {\tt johnstonjt@ornl.gov}}
\, and Matt McGinnis\footnote{Department of Mathematical Sciences, University of Delaware, Newark, DE 19716-2553, USA. Research supported by NSF DMS-1600768 grant. {\tt macmginn@udel.edu}} }
\date{\today}
\maketitle

\begin{abstract}
Let $n\geq k\geq 2$ be two integers and $S$ a subset of $\{0,1,\dots,k-1\}$.
The graph $J_{S}(n,k)$ has as vertices the $k$-subsets of the $n$-set $[n]=\{1,\dots,n\}$ and two $k$-subsets $A$ and $B$ are adjacent if $|A\cap B|\in S$.
In this paper, we use Godsil-McKay switching to prove that for $m\geq 0$, $k\geq \max(m+2,3)$ and $S = \{0, 1, ..., m\}$, the graphs $J_S(3k-2m-1,k)$ are not determined by spectrum
and for $m\geq 2$, $n\geq 4m+2$ and $S = \{0,1,...,m\}$ the graphs $J_{S}(n,2m+1)$ are not determined by spectrum.
We also report some computational searches for Godsil-McKay switching sets in the union of classes in the Johnson scheme for $k\leq 5$.
\end{abstract}

\section{Introduction}

The spectrum of a graph $G$ is the multi-set of eigenvalues of its adjacency matrix (see \cite{BH} for an introduction to spectral graph theory). Two graphs are called cospectral if they have the same spectrum. A graph $G$ is determined by spectrum if any graph cospectral to $G$ must be isomorphic to $G$. Two non-isomorphic graphs that are cospectral are called cospectral mates. An important research area of spectral graph theory is devoted to determining which graphs are determined by their spectra (see \cite{DH1,DH2} for example).

In this paper, we consider this problem for the union of classes in the Johnson association scheme. 
Let $n\geq k\geq 2$ be two integers and $S$ a subset of $\{0,1,\dots,k-1\}$.

The graph $J_{S}(n,k)$ has as vertices the $k$-subsets of the $n$-set $[n]=\{1,\dots,n\}$ and two $k$-subsets $A$ and $B$ are adjacent if $|A\cap B|\in S$.

Using this notation, the Johnson graph $J(n,k)$ is the graph $J_{\{k-1\}}(n,k)$ and the Kneser graph $K(n,k)$ is $J_{\{0\}}(n,k)$.

It is known that $J_{\{1\}}(n,2)$ is determined by its spectrum precisely when $n\neq 8$ (see \cite{Chang,Connor,Hoffman}). Also, the odd graphs $J_{\{0\}}(2k+1,k)$ for $k\geq 2$ are known to be determined by spectrum (see \cite{HL}). These graphs and their complements are the only nontrivial graphs in the Johnson scheme known to be determined by spectrum.

For $3\leq k\leq n-3$, the Johnson graphs $J(n,k)$ are not determined by spectrum (see \cite{DHKS}). For $k\geq 3$, the Kneser graphs $K(n,k)$ with $n=3k-1$ or $2n=6k-3+\sqrt{8k^2+1}$, as well as the mod-2 Kneser graph $J_{S}(n,k)$ (where $S$ is the set of even numbers in $\{0,1,\dots,k-1\}$) are not determined by spectrum (see \cite{HR}).

In this paper, we construct cospectral graphs with $J_{S}(n,k)$ for certain values of $n, k$ and $S$.
Our main results are that for any $m\geq0$ and $k\geq \max(m+2,3)$, if $S=\{0,1,\dots,m\}$, the graphs $J_S(3k-2m-1,k)$ are not determined by spectrum and for $m\geq 2$, $n\geq 4m+2$, if $S = \{0,1,...,m\}$, the graphs $J_{S}(n,2m+1)$ are not determined by spectrum.
Our main tool is the method of Godsil-McKay switching which constructs cospectral graphs to a given graph under certain regularity situations.
 
\begin{theorem}[Godsil-McKay \cite{GM}]\label{GM}
Let $G$ be a graph and $C_1\cup C_2 \cup \hdots \cup C_t \cup D$ a partition of the vertex set of $G$ satisfying the following conditions for $1\leq i,j\leq t$:
\begin{enumerate}[1)]
\item Any two vertices in $C_i$ have the same number of neighbors in $C_j$.
\item For each $C_i$ every vertex in $D$ is adjacent to $0, |C_i|/2$ or $|C_i|$ vertices in $C_i$. 
\end{enumerate}
Construct a new graph $H$ as follows. For every vertex $u$ not in $C_i$ with $|C_i|/2$ neighbors in $C_i$, delete the $|C_i|/2$ edges between $u$ and $C_i$ and join $u$ to the other $|C_i|/2$ vertices in $C_i$. Then $G$ and $H$ have the same spectrum.
\end{theorem}

The operation above that changes $G$ into $H$ is called Godsil-McKay switching. For our purposes we will only need $t=1$. Hence, from here forward we will simply use $C$ to denote a switching set.

At the end of our paper, we also report our computational results searching for switching sets in various classes and union of classes of the Johnson scheme including the Kneser graphs $K(9,3)$ and $K(10,3)$ which are the smallest examples of Kneser graphs for which it is not known whether they are determined by spectrum.

\section{Main Results}

\subsection{$J_S(n,2m+1)$ where $S = \{0, 1, ..., m\}$ and $n\geq 4m+2$}

\begin{theorem}\label{SS_2}
Let $m\geq 2$, $n\geq 4m+2$ and $S = \{0, 1, ..., m\}$.
If $G = J_S(n,2m+1)$ and $C =\{c\in V(G) : c\subset\{1,...,2m+2\} \text{ and } |c| = 2m+1\}$, then $C$ is a switching set in $G$.
\end{theorem}

\begin{proof}
First, note that $C$ is an independent set of size $\binom{2m+2}{2m+1} = 2m+2$ since the intersection of any two vertices in $C$ has cardinality $2m$. Let $R = [2m+2]$ and $u\in V(G)\setminus C$.  
We now consider the following cases:

\begin{enumerate}[(i)]
\item If $0\leq|u\cap R|\leq m$, then for $c \in C$, $|u\cap c|\leq m$. Hence $u$ will be adjacent to every vertex in $C$.

\item If $|u\cap R| = m+1$, then there are $\binom{(2m+2)-(m+1)}{m} = \binom{m+1}{m} = m+1$ vertices in $C$ that intersect $u$ in $m+1$ elements, implying they will not be adjacent to $u$.
The remaining $m+1$ vertices in $C$ must be adjacent to $u$ as their intersection with $u$ will have cardinality $m$.
Hence $u$ will be adjacent to exactly half of the vertices in $C$.

\item If $m+2\leq |u\cap R|\leq 2m$, then $|R\setminus u|\leq (2m+2)-(m+2) = m$.
So for $c\in C$, $|u\cap c|\geq m+1$.
Hence $u$ will have no neighbors in $C$.
\end{enumerate}
\end{proof}

\begin{theorem}\label{NonIso_2}
Let $m\geq 2$, $S = \{0, 1, ..., m\}$ and $C$ be the switching set described in Theorem \ref{SS_2}.
If $G = J_S(n,2m+1)$ and $CM$ is the graph obtained by switching with respect to $C$, then $G$ and $CM$ are not isomorphic.
\end{theorem}

We define the \textit{common neighbor count} $\lambda_G(x,y)$ of two vertices, $x$ and $y$ in $G$, as the number of vertices that are neighbors of both $x$ and $y$.
The \textit{common neighbor pattern} of a vertex, $x$ in $G$, is the multi-set of all possible values of $\lambda_G(x,y)$ where $y$ runs through the vertex set of $G$.

Consider the following vertices 
\begin{equation*}
c_0 = [2m+1] \text{ and } v = \{2m+2,...,4m+2\}
\end{equation*}
in $G$ and $CM$.

\begin{lemma}\label{CommonNeighbor2}
If $G$ and $CM$ are isomorphic, then $\lambda_G(c_0,v) = \lambda_{CM}(c_0,v)$.
\end{lemma}

\begin{proof}
As $G$ is vertex transitive, all vertices of $G$ will have the same common neighbor pattern. 
If $G$ and $CM$ are isomorphic, then all vertices of $CM$ will have the same common neighbor pattern as well.
In particular, $v$ will have the same common neighbor pattern before and after switching.
Based on the way switching is defined it follows that $\lambda_G(u,v) = \lambda_{CM}(u,v)$ for all $u$ not in $C$.
This implies
\begin{equation*}
\{\lambda_G(c_0,v),...,\lambda_G(c_{2m+1},v)\} = \{\lambda_{CM}(c_0,v),...,\lambda_{CM}(c_{2m+1},v)\},
\end{equation*}
where $c_i$ is the vertex in $C$ that does not contain the element $i$ for $1\leq i\leq 2m+1$.

Moreover, the permutation $(1,2,...,2m+1)$ of $[n]$ induces an automorphism of $G$ that fixes $c_0$ and $v$ and cyclically shifts $c_1,c_2,...,c_{2m+1}$.
It follows that
\begin{equation*}
\lambda_G(c_1,v) = \hdots = \lambda_G(c_{2m+1},v).
\end{equation*}
This permutation remains an automorphism after switching, thus
\begin{equation*}
\lambda_{CM}(c_1,v) = \hdots = \lambda_{CM}(c_{2m+1},v).
\end{equation*}
Therefore, if $G$ and $CM$ are isomorphic $\lambda_G(c_0,v) = \lambda_{CM}(c_0,v)$.
\end{proof}

\begin{proof}[Proof of Theorem \ref{NonIso_2}]
By Lemma \ref{CommonNeighbor2}, it is sufficient to show $\lambda_G(c_0,v) \neq \lambda_{CM}(c_0,v)$.
Before switching $c_0$ is adjacent to vertices of the form
\begin{equation*}
\binom{[2m+1]}{m}\cup\{2m+2\}\cup\binom{[n]\setminus[2m+2]}{m}.
\end{equation*}
Of these vertices, there are exactly $\binom{2m+1}{m}\left(\sum\limits_{i=0}^{m-1} \binom{2m}{i}\binom{n-(4m+2)}{m-i}\right)$ adjacent to $v$.
This accounts for the number of common neighbors of $c_0$ and $v$ lost during switching.
After switching $c_0$ becomes adjacent to vertices of the form
\begin{equation*}
\binom{[2m+1]}{m+1}\cup\binom{[n]\setminus[2m+2]}{m}.
\end{equation*}
Of these vertices, there are exactly $\binom{2m+1}{m+1}\left(\sum\limits_{i=0}^m \binom{2m}{i}\binom{n-(4m+2)}{m-i}\right)$ adjacent to $v$.
As $n\geq 4m+2$, it follows that $\lambda_{CM}(c_0,v) = \lambda_{G}(c_0,v)+\binom{2m+1}{m+1}\binom{2m}{m}$. 
Therefore, $G$ and $CM$ are nonisomorphic.
\end{proof}

\begin{corollary}\label{Non-DS2}
For $m\geq 2$, $n\geq 4m+2$ and $S = \{0, 1, ..., m\}$, the graphs $J_S(n,2m+1)$ are not determined by spectrum.
\end{corollary}

\subsection{$J_S(3k-2m-1,k)$ where $S = \{0,...,m\}$}

\begin{theorem}\label{SS_1}
Let $m\geq 0$, $k\geq m+2$ and $S = \{0, 1, ..., m\}$. If $G = J_S(3k-2m-1,k)$ and $C = \{c\in V(G) : [k-1]\subset c\}$, then $C$ is a switching set in $G$.
\end{theorem}

\begin{proof}
First, note that $C$ is an independent set. Let $R = [3k-2m-1]\setminus [k-1]$. Then $|R| = 2(k-m)$.

Now, every $c\in C$ has the form $[k-1]\cup \{r\}$ for some $r\in R$. So $|C| = 2(k-m)$. Let $u\in V(G)\setminus C$. We consider the following cases:

\begin{enumerate}[(i)]
\item If $2\leq |u\cap R|\leq k-m-1$, then $|u\cap [k-1]|\geq k-(k-m-1) = m+1$. Hence $u$ has no neighbors in $C$.

\item If $|u\cap R| = k-m$, then $|u\cap [k-1]| = m$. So there will be $k-m$ vertices in $C$ sharing exactly $m$ elements with $u$ and $k-m$ vertices in $C$ sharing exactly $m+1$ elements with $u$. Hence, $u$ will be adjacent to half the vertices in $C$.

\item If $k-m+1\leq |u\cap R|\leq k$, then $|u\cap [k-1]|\leq k-(k-m+1)=m-1$. Hence $u$ will be adjacent to each vertex in $C$.
\end{enumerate}
\end{proof}

\begin{theorem}\label{NonIso_1}
Let $m\geq 0$, $k\geq \max(m+2,3)$, $S = \{0, 1, ..., m\}$ and $C$ be the switching set described in Theorem \ref{SS_1}.
If $G = J_S(3k-2m-1,k)$ and $CM$ is the graph obtained by switching with respect to $C$, then $G$ and $CM$ are not isomorphic.
\end{theorem}

Consider the vertices 
\begin{equation*}
c_0 = [k],\; c_1 = [k-1]\cup\{k+1\} \text{ and } w = [k-2]\cup\{k,k+1\}
\end{equation*}
in $G$ and $CM$.

\begin{lemma}\label{CommonNeighbor1}
If $G$ and $CM$ are isomorphic, then $\lambda_G(c_0,w)= \lambda_{CM}(c_0,w)$.
\end{lemma}

\begin{proof}
As $G$ is vertex transitive, all vertices of $G$ will have the same common neighbor pattern. 
So if $G$ and $CM$ are isomorphic, then all vertices of $CM$ will have the same common neighbor pattern.
In particular, $w$ will have the same common neighbor pattern before and after switching.
Based on the way switching is defined it follows that $\lambda_G(u,w) = \lambda_{CM}(u,w)$ for all $u$ not in $C$.

Hence,
\begin{equation*}
\{\lambda_G(c_0,w),...,\lambda_G(c_{2k-2m-1},w)\} = \{\lambda_{CM}(c_0,w),...,\lambda_{CM}(c_{2k-2m-1},w)\},
\end{equation*}
where $c_i = [k-1]\cup \{k+i\}$ for $0\leq i\leq 2k-2m-1$.

The permutation $(k+2,...,3k-2m-1)$ of $[3k-2m-1]$ induces an automorphism of $G$ that fixes $c_0$, $c_1$ and $w$ and cyclically shifts $c_2,...,c_{2k-2m-1}$.

It follows that
\begin{equation*}
\lambda_G(c_2,w) = \hdots = \lambda_G(c_{2k-2m-1},w).
\end{equation*}
This permutation remains an automorphism after switching, thus
\begin{equation*}
\lambda_{CM}(c_2,w) = \hdots = \lambda_{CM}(c_{2k-2m-1},w).
\end{equation*}
Hence, if $G$ and $CM$ are isomorphic $\{\lambda_G(c_0,w),\lambda_G(c_1,w)\} = \{\lambda_{CM}(c_0,w),\lambda_{CM}(c_0,w)\}$. Observing $\lambda_G(c_0,w) = \lambda_G(c_1,w)$ gives the desired result.
\end{proof}

\begin{proof}[Proof of Theorem \ref{NonIso_1}]
By Lemma \ref{CommonNeighbor1}, it is sufficient to show $\lambda_G(c_0,w) \neq \lambda_{CM}(c_0,w)$. Before switching $c_0$ is adjacent to vertices of the form
\begin{equation*}
\binom{[k-1]}{m}\cup\binom{[3k-2m-1]\setminus c_0}{k-m}.
\end{equation*}
Of these vertices, there are exactly $\binom{k-2}{m}\binom{2(k-m-1)}{k-m} + \binom{k-2}{m-1}\binom{2k-2m-1}{k-m}$ adjacent to $w$.
This accounts for the number of common neighbors of $c_0$ and $w$ deleted during switching.
After switching $c_0$ becomes adjacent to vertices of the form
\begin{equation*}
\binom{[k-1]}{m}\cup\{k\}\cup\binom{[3k-2m-1]\setminus c_0}{k-m-1}.
\end{equation*}
Of these vertices, there are exactly $\binom{k-2}{m-1}\binom{2(k-m-1)}{k-m-1}$ adjacent to $w$.
As $k \geq m+2$, it follows that $\binom{k-2}{m}\binom{2(k-m-1)}{k-m} + \binom{k-2}{m-1}\binom{2k-2m-1}{k-m} > \binom{k-2}{m-1}\binom{2(k-m-1)}{k-m-1}$. Hence, $\lambda_{G}(c_0,w) > \lambda_{CM}(c_0,w)$.
Therefore, $G$ and $CM$ are not isomorphic.
\end{proof}

\begin{corollary}\label{Non-DS1}
For $m\geq 0$, $k\geq \max(m+2,3)$, $S = \{0, 1, ..., m\}$, the graphs $J_S(3k-2m-1,k)$ are not determined by spectrum.
\end{corollary}

\section{Relation to Previous Work}
\subsection{Kneser and Johnson Graphs}

Taking $m = 0$ in Theorem \ref{SS_1}, we obtain the switching sets found for Kneser graphs $K(3k-1,k)$ in \cite{HR}.

In addition to this generalization, one may also notice that the switching sets described in Theorem \ref{SS_1} are a generalization of the switching sets found for the Johnson graphs $J(n,3)$ in \cite{DHKS}. Indeed, a switching set for $J(n,3)$ can be obtained by taking the 4 vertices being $3$ element subsets of a set of size $4$. We note here $J_S(n,k)$ is isomorphic with $J_{S+n-2k}(n,n-k)$ and with the complement of $J_{K\setminus S}(n,k)$, where $K = \{0,1,...,k-1\}$. Using these facts, it follows that the complement of $J_{\{2\}}(n,3)$ is $J_{\{0,1\}}(n,3)$ and $J_{\{0,1\}}(n,3)$ is isomorphic to $J_{\{n-6,n-5\}}(n,n-3)$. Thus, this family of graphs can be described in Theorems \ref{SS_1} and \ref{NonIso_1} by letting $m = k-2$.

What is more, the switching sets for $J(n,3)$ were extended for $k > 3$ by using the more general form of switching described in Theorem \ref{GM}. A switching partition can be constructed for $J(n,k)$, $n-3\geq k\geq 3$ by fixing a set $Y$ of size 4 and letting $D$ be the set of all $k$-sets that do not contain precisely 3 elements from $Y$. Then for each $(k-3)$-subset, $i$, of $[n]\setminus Y$, we allow $C_i$ to be the set of four $k$-sets containing $i$ and precisely 3 elements from $Y$.

This leads to the question of whether or not we can make a similar generalization in Theorems \ref{SS_1} and \ref{NonIso_1}.

In an attempt to extend our results for $n\neq 3k-2m-1$ we construct similar sets $C_i$ in the following way. For each $(k-1)$-subset, $i$, of $[n-2(k-m)]$ take $C_i$ to be the set of $2(k-m)$ vertices containing $i$ and one element from $[n]\setminus [n-2(k-m)]$. When $m = k-2$ we obtain switching sets for the complements of the Johnson graphs $J(n,k)$, $n-3\leq k\leq 3$.

Now, consider the graph $J_S(3k-2m-1,k)$ described in Theorems \ref{SS_1} and \ref{NonIso_1} and take $v$ to be a vertex containing $[m]\cup\{k\}$ and $k-m-1$ elements from $[n]\setminus [n-2(k-m)]$. Consider the switching set $C_1$ formed by taking the $2(k-m)$ vertices containing $[k-1]$ and one element from $[n]\setminus [n-2(k-m)]$. It is easily seen that $v$ will have $k-m+1$ neighbors in $C_1$. As $k-m<k-m+1<2(k-m)$ it follows that $v$ cannot be in $D$. Thus, the only option is for $v$ to be in some other $C_i$ which can only occur if $m = k-2$.

\subsection{Vertices containing $[k-2]$ as a possible switching set}

In \cite{HR} it was shown that taking $C$ to be the set of vertices containing $[k-2]$ as a subset is a switching set for the graphs $K(n,k)$ satisfying $2n = 6k-3+\sqrt{8k^2+1}$. It is reasonable to try to generalize this in a way similar to what we have done for $K(3k-1,k)$.

Consider the graph $J_S(n,k)$ where $S = \{0,...,m\}$ and let $C$ be the set of vertices containing $[k-2]$ as a subset. Suppose $u$ is a vertex not in $C$. We consider the following cases:

\begin{enumerate}[(i)]
\item If $u$ has $m+1$ or more elements in $[k-2]$, then $u$ is adjacent to no vertices in $C$.

\item If $u$ has $m-2$ or less elements in $[k-2]$, then $u$ is adjacent to all vertices in $C$.

\item If $u$ has $m-1$ elements in $[k-2]$, then $u$ is adjacent to $\binom{n-k+2}{2}-\binom{k-m+1}{2}$ vertices in $C$.

\item If $u$ has $m$ elements in $[k-2]$, then $u$ is adjacent to $\binom{n-k+2}{2}-(n-2k+m+2)(k-m)-\binom{k-m}{2}$ vertices in $C$.
\end{enumerate}

So our restrictions on $n$, $k$ and $m$ will come from (iii) and (iv). Note that $\binom{n-k+2}{2}-\binom{k-m+1}{2} \neq |C|$ as $m<k-1$. If $\binom{n-k+2}{2}-\binom{k-m+1}{2} = 0$, then $n=2k-m-1$, but evaluating (iv) with $n=2k-m-1$ gives $0$ so this would not give a desirable switching set.

If $\binom{n-k+2}{2}-\binom{k-m+1}{2} = (1/2)|C|$, then solving for $n$ we obtain 
\begin{equation*}
n = (1/2)(\sqrt{8k^2-16km+8k+8m^2-8m+1} + 2k-3). 
\end{equation*}

If (iv) is equal to $|C|$, then we find $n = (1/2)(3k-m-3)$. Setting these equal gives no solutions. 

If (iv) is equal to 0, then we find $n = 2k-m-1$ or $n = 2k-m-2$, which will make (iii) equal to $0$ or $m-k$, respectively. 

Finally, if (iv) is equal to $(1/2)|C|$, then (iii) and (iv) are equal and we find $n = 2k-m+1$. Solving for $k$ we find $k = (1/2)(2m-\sqrt{33}+3)$ or $k = (1/2)(2m+\sqrt{33}+3)$, neither of which are integers. 

One thing to note is when $m=0$, (iii) is not a possibility and we need (iv) equal to $(1/2)|C|$. Solving for $n$ in this case we obtain $2n = 6k-3 + \sqrt{8k^2+1}$, the same parameters found previously for Kneser graphs.

\section{Computational Results}

In this section, we report on a list of classes and union of classes in the Johnson Scheme we have checked by computer for switching sets. We used two pieces of code to search for switching sets.
The first code made use of GPUs (general purpose Graphical Processing Units) to exhaustively search a graph for (small) switching sets.
The key feature of GPUs is that they allow for massive parallel computation.
In our case, each independent thread examines one induced subgraph of size $2k$ for $k=2, 3, 4, 5, ...$ specified by the user.
Because graphs in the Johnson scheme are vertex transitive we were able to reduce the necessary computation and only examine subgraphs that included vertex $1$.
While the GPU dramatically speeds up the computation, if either the Johnson graph is large or the size of the subgraphs being examined is large then the computation is still prohibitively long. Focusing mainly on Kneser graphs, we were able to eliminate the possibility of switching sets of size 8 in $K(9,3)$, $K(10,3)$, $K(11,3)$, $K(12,3)$ and $K(10,4)$ as well as switching sets of size 10 in $K(9,3)$ and $K(10,3)$. Our computations extend the computations of Haemers and Ramezani \cite{HR} which did not find any switching sets of size $4$ or $6$ in the Kneser graphs $K(9,3)$ nor $K(10,3)$. At present time, these are smallest graphs in the Johnson scheme whose spectral characterization is not known.

The second code employed a technique similar to backtracking and searched only for switching sets of size $4$ (an independent set, an induced matching, an induced cycle, and a complete graph),
and switching sets of size $6$ restricted to independent sets, induced matchings, and an induced $6$-cycle.
Because of the restrictions on the type of switching sets that were searched for (and the relatively small size) we were able to explore larger graphs in the Johnson scheme.
Both codes are available on github at \url{https://github.com/jtjohnston/computational_combinatorics/tree/master/GM-switching}.

We describe below the notation used in the subsequent tables:
\begin{itemize}
\item 0b indicates that no switching sets were found using backtracking technique.
\item 0eX indicates that no switching sets were found of size 4, 6, ..., X using the exhaustive search on GPU.
\item 1(DS) indicates that these graphs have already been proven to be DS.
\item 1(NDS) indicates that these graphs have already been proven to not be DS.
\item 1+ indicates we found a new switching set and the graph after switching is non-isomorphic.
\item 1- indicates we found a new switching set but the graph after switching is isomorphic.
\end{itemize}

\begin{center}
Table for $k=3$
\end{center}
\begin{center}
\begin{tabular}{ |c|c|c|c|c|c|c| } 
 \hline
 	$n$ & $S=\{0\}$ & $S=\{1\}$ & $S = \{2\}$ 	& $S = \{0,1\}$ & $S = \{0,2\}$ & $S = \{1,2\}$ \\ 
 \hline
 	6 	& 1(DS) 	& 1(NDS) 	& 1(NDS) 		& 1(NDS) 		& 1(NDS) 		& 1(DS)   \\ 
 	7 	& 1(DS) 	& 1(NDS) 	& 1(NDS) 		& 1(NDS) 		& 1(NDS) 		& 1(DS)  \\ 
 	8 	& 1(NDS) 	& 1(NDS) 	& 1(NDS) 		& 1(NDS) 		& 1(NDS) 		& 1(NDS) \\
 	9 	& 0e10		& 1(NDS) 	& 1(NDS) 		& 1(NDS) 		& 1(NDS) 		& 0e10	 \\
 	10 	& 0e10 		& 1(NDS) 	& 1(NDS) 		& 1(NDS) 		& 1(NDS) 		& 0e8    \\
 	11 	& 0e8 		& 1(NDS) 	& 1(NDS) 		& 1(NDS) 		& 1(NDS) 		& 0e6    \\
 	12 	& 0e8 		& 1(NDS) 	& 1(NDS) 		& 1(NDS) 		& 1(NDS) 		& 0b     \\
 	13 	& 0b 		& 1(NDS) 	& 1(NDS) 		& 1(NDS) 		& 1(NDS) 		& 0b     \\
 	14 	& 0b 		& 1(NDS) 	& 1(NDS) 		& 1(NDS) 		& 1(NDS) 		& 0b     \\
 	15 	& 0b 		& 1(NDS) 	& 1(NDS) 		& 1(NDS) 		& 1(NDS) 		& 0b     \\
 \hline
\end{tabular}
\end{center}

\begin{center}
Table for $k=4$
\end{center}
\begin{center}
\begin{tabular}{ |c|c|c|c|c|c|c|c| } 
 \hline
	$n$ & $S=\{0\}$ & $S=\{1\}$ 	& $S = \{2\}$ 	& $S = \{3\}$ 	& $S = \{0,1\}$ & $S = \{0,2\}$ & $S = \{1,2\}$ \\ 
 \hline 
 	8 	& 1(DS) 	& 0e8 			& 1- 			& 1(NDS) 		& 0e8 			& 1(NDS) 		& 0e8 \\
 	9 	& 1(DS)	 	& 0e8 			& 0e8 			& 1(NDS) 		& 1+ 			& 1(NDS) 		& 0e8 \\
 	10 	& 0e8 		& 0b 			& 0b 			& 1(NDS) 		& 0b 			& 1(NDS) 		& 0b  \\
 	11 	& 1(NDS) 	& 0b 			& 0b 			& 1(NDS) 		& 0b 			& 1(NDS) 		& 0b  \\
 	12 	& 0e6 	& 0b 			& 0b 			& 1(NDS)		& 0b 			& 1(NDS) 		& 0b  \\
 \hline
\end{tabular}
\end{center}

\begin{center}
Table for $k=5$
\end{center}

\begin{center}
\begin{tabular}{ |c|c|c|c|c|c|c|c| } 
 \hline
	$n$ & $S=\{0\}$ 	& $S=\{1\}$ 	& $S = \{2\}$ 	& $S=\{3\}$ & $S = \{4\}$ \\ 
 \hline 
 	10 	& 1(DS) 		& 0e6 			& 0e6 			& 0e6		& 1(NDS) \\
 	11 	& 1(DS)			& 0b 			& 0b 			& 0b 		& 1(NDS) \\
 \hline
\end{tabular}
\end{center}

\section{Open Problems}
We propose the following as future research.

\begin{enumerate}[1)]
\item Determine whether or not other graphs in the Johnson scheme are DS. As mentioned above, the smallest open case as of now is $K(9,3)$.
\item Regarding the tables above, we were only able to find two new graphs with potential switching sets. For $J_{\{0,1\}}(9,4)$ we were able to generalize in Theorem \ref{SS_1}.

The other graph we found switching sets for is $J_{\{2\}}(8,4)$, however none of these switching sets produced nonisomorphic cospectral mates. However, it is interesting to note that the switching sets we found were of size 4. Focusing more closely on this graph we extended our computations to look for switching sets of size 8. We found two different switching sets of size 8 that produced nonisomorphic cospectral mates. The switching sets are
\end{enumerate}
\begin{equation*}
C = \{\{1,2,3,4\}, \{1,2,5,6\}, \{1,2,3,5\}, \{1,2,4,6\}, \{3,4,7,8\}, \{3,5,7,8\}, \{4,6,7,8\}, \{5,6,7,8\}\}
\end{equation*}

which is the union of two 4-cycles. The second is a 6-regular graph on 8 vertices
\begin{equation*}
C = \{\{1,2,3,4\}, \{1,2,3,5\}, \{1,4,6,7\}, \{1,5,6,7\}, \{2,3,4,8\}, \{2,3,5,8\}, \{4,6,7,8\}, \{5,6,7,8\}\}.
\end{equation*}

It would be nice to see if these sets can be generalized to produce another infinite family of cospectral mates in the Johnson scheme.

\section*{Notice of Copyright}
\begin{small}
This manuscript has been authored by UT-Battelle, LLC under Contract No. DE-AC05-00OR22725 with the U.S. Department of Energy.
The United States Government retains and the publisher, by accepting the article for publication, acknowledges that the United
States Government retains a non-exclusive, paid-up, irrevocable, worldwide license to publish or reproduce the published form
of this manuscript, or allow others to do so, for United States Government purposes. 
The Department of Energy will provide public access to these results of federally sponsored research in accordance with the 
DOE Public Access Plan (http://energy.gov/downloads/doe-public-access-plan).
\end{small}

\end{document}